\documentclass{ws-jbs}
\usepackage[super,compress]{cite}
\usepackage{url}
\usepackage{caption}
\usepackage{subcaption}
\usepackage{xcolor}
\newtheorem{teo}{Theorem}[section]

\newtheorem{pro}[teo]{Proposition}

\newtheorem{ejem}{Example}
\usepackage{graphicx}
\usepackage{amsmath}
\usepackage{amssymb}
\usepackage{appendix}

\begin{document}

\markboth{Jim\'enez, Bl\'e, Sep\'ulveda-Quiroz, Mart\'inez-Garc\'ia, \'Alvarez-Gonz\'alez}
{A mathematical model for captive fish with cannibalism}


\title{A MATHEMATICAL MODEL FOR CAPTIVE FISH WITH CANNIBALISM.}

\author{MAR\'IA FERNANDA JIM\'ENEZ}
\address{Tecnol\'ogico Nacional de México-Instituto Tecnol\'ogico Superior de Comalcalco\\ Carretera Vecinal Comalcalco-Paraiso Km. 2, Rancheria Occidente 3ra. Secci\'on\\ Comalcalco, Tabasco, c.p. 86651, M\'exico\,\\
\email{mfernanda\_jimenez@hotmail.com}
}

\author{GAMALIEL BL\'E\footnote{Corresponding author.}}
\address{Divisi\'on Acad\'emica de Ciencias B\'asicas, Universidad Ju\'arez Aut\'onoma de Tabasco\\ Carretera Cunduac\'an-Jalpa de M\'endez, Km. 1\\ Cunduac\'an, Tabasco, c.p. 86690, M\'exico.\,\\
	\email{gble@ujat.mx}
}

\author{C\'ESAR ANTONIO SEP\'ULVEDA-QUIROZ}
\address{Tecnol\'ogico Nacional de M\'exico-Instituto Tecnol\'ogico de Villahermosa\\ Carretera Villahermosa-Frontera, Km. 3.5, Ciudad Industrial\\ Villahermosa, Tabasco, c.p. 86010, M\'exico.\,\\
	\email{casq15@gmail.com}
}

\author{RAFAEL MART\'INEZ-GARC\'IA$^{+}$}
\author{CARLOS ALFONSO \'ALVAREZ-GONZ\'ALEZ$^{++}$}
\address{Laboratorio de Fisiolog\'ia en Recursos Acu\'aticos (LAFIRA), \\Divisi\'on Acad\'emica de Ciencias Biol\'ogicas, Universidad Ju\'arez Aut\'onoma de Tabasco, Carretera Villahermosa-C\'ardenas Km 0.5 \\Villahermosa, Tabasco, C.P.86039, M\'exico.\,\\
	\email{biologomartinez@hotmail.com$^{+}$}
		\email{alvarez\_alfonso@hotmail.com$^{++}$}
}

\maketitle


\begin{abstract}
A mathematical model is analyzed for a population of fish exhibiting cannibalism during its larval stage. The population is divided into three classes (vulnerable, cannibals, and non-vulnerable). It is assumed that initially, the entire population is vulnerable, and as they develop, individuals transition into one of the other classes. In addition, the interaction between the non-vulnerable and the cannibals is a predator-prey relationship with a Beddington-DeAngelis functional response. The global stability of the equilibrium point is demonstrated and numerical simulations show how food quality influences the cannibalistic behaviour of the population and its survival. Furthermore, it is shown that the solutions fit the data obtained experimentally for \textit{Atractosteus tropicus} in its larval state. This model can be adapted for other fish species that exhibit cannibalism.
\end{abstract}


\section{Introduction}

South America has a diverse fish fauna, there are more than 5,000 species of freshwater fish, which represents one-third of the world's freshwater fish species \cite{Re}. Since fish is a dietary staple for humans, {it has lead to the expansion} of breeding of fish in natural or artificial ponds, and this is why the aquaculture was born. The success of aquaculture as an industry depends on advancements in larviculture \cite{Ga}, which aims to increase survival rates during the transition from larvae to juveniles, through appropriate environmental conditions and a proper feeding. However, one of the principal causes of mortality during this stage is cannibalism \cite{Sw,Ga}.\\
Cannibalism is the act or practice of consuming members of the same species \cite{Ba}. This behavior is a specific type of predation that involves attacking to its conspecific to partially or entirely consume it \cite{Sm,Pf,Ga}. Cannibalism is common in fish and is considered one of the main causes of mortality in the larval stage \cite{Do,Ba,Ga}. Approximately 390 fish species have been recorded to exhibit some form of cannibalism, of which only 150 display this behavior in captivity. In most species, cannibalism tends to occur throughout their lifetime, while in others, it is only observed during the early stages of development (larvae or juveniles) \cite{Pe}. Cannibalism is often associated with size variation, limited food availability, high population density, lack of shelter areas, and lighting conditions, with the first two being the primary causes \cite{He,Qi,Sm,HP}. There are seven types of cannibalism, which are distributed across three criteria: the prey's developmental stage, the genetic relationship of the cannibal to the prey, and the age relationship between the cannibal and the prey. In the developmental stage criterion, there are egg cannibalism and post-hatching cannibalism. In the genetic relationship criterion, there are filial cannibalism, sibling cannibalism, and non-kin cannibalism. In the age relationship criterion, we find intracohort and intercohort cannibalism  \cite{Sm}. For example,  \textit{Cyprinodon diabolis} and \textit{Coleomegilla maculata lengi} exhibit egg cannibalism \cite{Bu,Gg}, \textit{Lates calcarifer, Centropomus undecimalis} and \textit{Atractosteus tropicus} display post-hatching and intracohort cannibalism \cite{Go,Ha,Se}, \textit{Gasterosteus aculeatus} is a fish that shows filial and intercohort cannibalism \cite{Ro}, \textit{Brycon moorei} and \textit{Cyprinus carpio} are fish that exhibit sibling cannibalism \cite{BN,Va}. Cannibalism is considered a feeding strategy that ensures the survival of a species since it reduces intraspecific competition when there is limited food availability \cite{Lo}. On the other hand, it has been observed that cannibalism can be controlled through proper management of food supply \cite{He,Pi}. \\
{
	The tropical gar (\textit{Atractosteus tropicus}) is a species cataloged as ancestral that lives in Central American countries and in southeastern Mexico. This species is considered of great economic and cultural importance in the region \cite{Mc1,N}. It is the ideal fish for sustainable tropical aquaculture, as it grows rapidly, thrives in low dissolved oxygen environments, tolerates high densities at all stages of cultivation, is carnivorous but adapts well to balanced feed, and can be cultivated with other fish to reduce production costs. Moreover, it is disease-resistant, and cannibalism only occurs in the larval stage and to a lesser extent in the juvenile stage \cite{Mc}. The tropical gar is not born cannibal, but some become cannibals in their larval or juvenile stage. Some factors that have been reported in relation to cannibalism in the larval stage are culture density, variability in size or color between larvae, insufficient feeding or low nutritional value, and the presence of abnormally larvae \cite{Mc}. Therefore, it has been important to implement strategies to reduce cannibalism and increase survival. These strategies include managing the cannibals and adjusting the diet \cite{F,Mc1}. Experiments carried out to measure the effect of feeding on cannibalism showed that in the larval stage a survival of 32\% can be achieved \cite{P}. Since cannibalism in the \textit{Atractosteus tropicus} occurs mainly in the first days of its larval state, the goal of this work is to propose a mathematical model based on differential equations, consistent with the experimental results, that contributes to improving the survival strategies of this species in its larval stage.}

Cannibalism involves an intraspecific predator-prey interaction \cite{Fo,Po}. Mathematical predator-prey models that include intraspecific competition have been used to model it \cite{Al,CD,Cl,Di,Ko,La,Li,Ma,Ta,Za}. For example, Al Basheer, \textit{et al.}  proposed a Holling-Tanner type predator-prey model considering cannibalism in both the prey and predator populations. Specifically, the model is 

\begin{equation*}
	\begin{aligned}
		\dfrac{du}{dt}&=u(1+c_1-u)-\dfrac{uv}{u+\alpha v}-c\left(\dfrac{u^2}{u+d}\right),\\
		\dfrac{dv}{dt}&=\delta v\left(\beta-\dfrac{v}{\gamma u+\rho v}\right),
	\end{aligned}
\end{equation*}

where $u$ and $v$ represent the prey and predator populations, respectively. In this case, $c\left(\dfrac{u^2}{u+d}\right)$ models cannibalism in the prey, $c_1$ represents the benefit obtained by the prey through cannibalism, and $\rho v$ denotes cannibalism in the predators \cite{Al}. They showed that depending on the parameters, unstable equilibrium points present when cannibalism is in a single population can become stable when cannibalism is present in both populations. On the other hand, Takyi E. {\it et al.} presented a predator-prey model, separating the prey population into two groups, juveniles and adults. They assumed the presence of cannibalism in the juvenile population. They showed that cannibalism has both a stabilizing and destabilizing effect depending on the choice of parameters \cite{Ta}.\\

Kaewmanee and Tang analyzed cannibalism within a predator-prey system. They divided the predator population into two classes (adult and juvenile) and assumed that cannibalism occurs from the adult class towards the juvenile class. They proved that there are parameter conditions under which the system has a stable coexistence equilibrium point. Furthermore, they showed that as the cannibalism rate increases and reaches a threshold value, this equilibrium point loses its stability, leading to the extinction of the predator population \cite{KT}.\\

In order to establish a system of differential equations that models the density in a pond of a population of \textit{Atractosteus tropicus} in its larval state, we are going to divide the population into three subclasses, the cannibal population ($y$), the vulnerable population ($x$), and the non-vulnerable population ($z$). Additionally, we establish the following hypotheses.
\begin{enumerate}
	\item {We assume that the entire population of larvae is vulnerable and as the days go by it can move to the cannibal class or develop and move to the non-vulnerable class.}
	\item The cannibal population is composed of those larvae that are larger than the vulnerable ones or have morphological characteristics suitable for consuming another of their own species.
	\item The non-vulnerable population consists of those larvae that develop and have morphological conditions that prevent them from being consumed and consuming individuals of their own species.
	\item {The vulnerable population and the cannibal population have a predator-prey relationship with a Beddington-DeAngelis type functional response}.
\end{enumerate}
Explicitly, the model is

\begin{equation}
	\begin{aligned}
		\dot{x}&=-\dfrac{(1-\alpha_1)\beta x y}{1+x+ay}-(1-\alpha_1)   (1-m)\beta x-\rho_1 x-\mu x,\\
		\dot{y}&=(1-\alpha_1)(1-m) \beta x-\alpha_1 \gamma  y^2-\rho_2 y-\mu  y,\\
		\dot{z}&=\alpha_1 m (n-x-z)(\rho_1 x+\rho_2 y)-\mu z,
	\end{aligned}
\end{equation}\label{modelo}
where
\begin{itemize}
	\item $x$ is the proportion of the vulnerable population.
	\item $y$ is the proportion of the cannibal population.
	\item $z$ is the proportion of the non-vulnerable population.
	\item $m$ is the parameter that measures the management time, {it is the proportion of time spent supervising and isolating cannibals, $m\in (0,1)$}.
	\item $a$ is the parameter that measures the interference among the cannibals.
	\item $n$ is the initial vulnerable population.
	\item $\alpha_1$ is the parameter that measures the quality of the administered food, $\alpha_1 \in (0,1)$.
	\item $\mu$ is the mortality rate of larval \textit{Atractosteus tropicus}.
	\item $\beta$ is the rate at the vulnerable population becomes cannibal.
	\item {$\rho_1$ and $\rho_2$ are the rates at which vulnerable and cannibal populations become non-vulnerable, respectively}.
	\item $\gamma$ represents the cannibalism within the cannibal population.
	\item The function $\dfrac{(1-\alpha_1)\beta x }{1+x+ay}$ measures the attack of the cannibals (predator) on the vulnerable population (prey). This takes into account satiation and interference in the predator.
	\item {The factor $n-x-z$ in the population growth rate $ z $ determines the maximum growth threshold of the non-vulnerable population.}
\end{itemize}
{We assume that all parameters of the system (\ref{modelo}) are positive and we are interested in the behavior of the solutions in the first units of time.  
}

\section{Stability Analysis}
For the analysis of the system (\ref{modelo}), we remember that $m,\alpha_1\in(0,1).$ \\
{ Let $$T=\left\lbrace (x,y,z)\in \mathbb{R}^3\, :\, 0 \leq x \leq n,\, 0 \leq y \leq n,\, 0 \leq z \leq n \right\rbrace. $$
	\begin{pro}\label{CPI}
		If $\beta\leq\dfrac{n\gamma \alpha_1+\mu+\rho_2}{(1-m)(1-\alpha_1)}$ then $T$ is positively invariant for system (\ref{modelo}).
	\end{pro}
	The proof of this result can be seen in Appendix \ref{A1}.}

Since  system (\ref{modelo}) models the interaction between  $x$, $y$, and $z$, we will study the solutions that are entirely in the positive octant.
Solving sytem (\ref{modelo}) we have that the only equilibrium point in this octact is $P_0=(0,0,0)$ and the Jacobian matrix evaluate at $P_0$ is
$$J(P_0)=\begin{pmatrix}
	-(1-m)(1-\alpha_1)\beta-\mu-\rho_1&0&0\\
	(m-1)(\alpha_1-1)\beta&-\mu-\rho_2&0\\
	mn\alpha_1\rho_1&mn\alpha_1\rho_2&-\mu
\end{pmatrix}.$$
The eigenvalues corresponding to $J(P_0)$ are
$$\lambda_1=-(1-m)(1-\alpha_1)\beta-\mu-\rho_1,\,\quad\lambda_2=-\mu-\rho_2,\quad \lambda_3=-\mu.$$
Thus, we have the following proposition.
\begin{pro}\label{EL}
	The trivial equilibrium point $P_0$  is locally asymptotically stable.
\end{pro}
{In fact, we prove the following result in Appendix \ref{A2}. } 
\begin{teo}\label{GS}
	The equilibrium point $P_0$ is globally asymptotically stable.
\end{teo}

Although all the solutions of system (\ref{modelo}) tend asymptotically to the origin, the important thing of the model is that in the first units of time, the solutions conform to the experimental data, corresponding to the \textit{Atractosteus tropicus} population in the first fourteen days of their larval state, as we will show in the next section. In addition, the model can be adapted for populations with similar behaviors.

\section{Numerical Simulations}
{Experimentally, it has been observed that cannibalism occurs after the eighth day of the larval state of the \textit{Atractosteus tropicus}, see Fig \ref{datos_reales}. Therefore, the survival curve is constant in the first 8 days, and simulations reflect what happens from the ninth day when cannibalism appears in the population. In the simulations, we assume that we start with 100\% of the larvae in the pond ($n=100$), and we will take the initial condition as $(100,0,0)$, which corresponds to initially having all larvae vulnerable.
	In this section, we show different scenarios of population behavior through numerical simulations.}
\begin{ejem}\label{Ejem_interferencia}
	Taking  $n=100$, $m=0.5$, $\alpha_1=0.5$, $\beta=1.69$, $\gamma=0.39$, $\mu=0.002$, $\rho_1=0.013$, and $\rho_2=0.018$ in  system (\ref{modelo}) and  considering no interference between the cannibals, we obtain a survival rate of {30\%}, as is shown in Figure \ref{a0}. However, if we consider an interference value of $a=2$, meaning high interference among the cannibals, we observe that survival increases but is still less than {40\%}, as shown in Figure \ref{a2}. From  simulations,
	we observe that by varying the interference we do not obtain survival rates greater than {40\%}.
	\begin{figure}[h]
		\centering
		\begin{subfigure}[b]{0.49\textwidth}
			\centering
			\includegraphics[width=\textwidth]{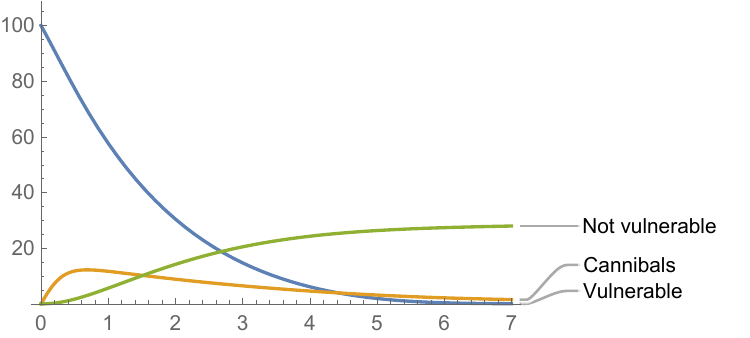}
			\caption{$a=0$}
			\label{a0}
		\end{subfigure}
		\begin{subfigure}[b]{0.49\textwidth}
			\centering
			\includegraphics[width=\textwidth]{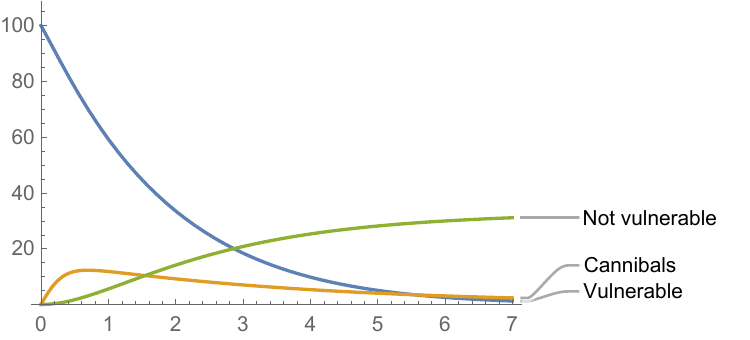}
			\caption{$a=2$}
			\label{a2}
		\end{subfigure}
		\caption{Time series of system (\ref{modelo}) with variation in cannibal interference.}
		\label{graficas_interferencia}
	\end{figure}
\end{ejem}

\begin{ejem}\label{Ejem_alimento}
	Fixing $a=2$, $n=100$, $m=0.5$, $\beta=1.69$, $\gamma=0.39$, $\mu=0.002$, $\rho_1=0.013$, and $\rho_2=0.018$ in system (\ref{modelo}). If we consider that the food is very unpalatable, i.e., $\alpha_1=0.05$, it can be observed that the survival rate is less than {10\%}, as is shown in Figure \ref{alp0.05}. While considering moderately acceptable food with $\alpha_1=0.5$, the survival rate doubles, reaching {30\%}, as shown in Figure \ref{alp0.5}. If we consider a food acceptance with $\alpha_1=0.7$, we obtain a higher survival rate of over {40\%}, as shown in Figure \ref{alp0.7}. These numerical results indicate that investing in the nutrition of \textit{Atractosteus tropicus} has a significant impact on reducing cannibalism within the species and increases the survival of the species, see Figure \ref{graficas_alimento}.
	\begin{figure}[h]
		\centering
		\begin{subfigure}[b]{0.49\textwidth}
			\centering
			\includegraphics[width=\textwidth]{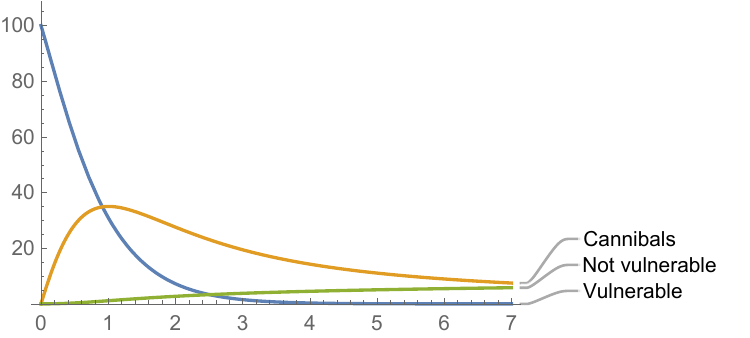}
			\caption{$\alpha_1=0.05$}
			\label{alp0.05}
		\end{subfigure}
		\begin{subfigure}[b]{0.49\textwidth}
			\centering
			\includegraphics[width=\textwidth]{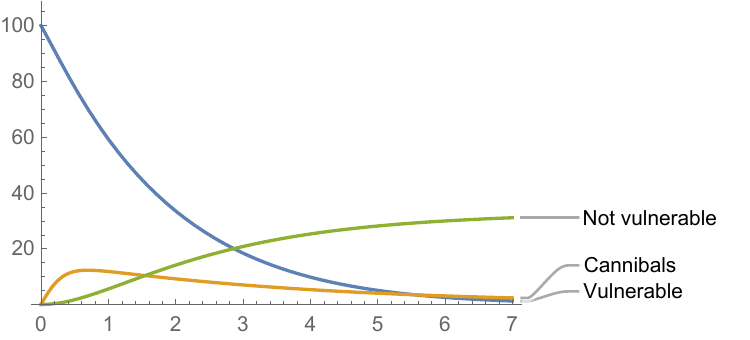}
			\caption{$\alpha_1=0.5$}
			\label{alp0.5}
		\end{subfigure}\vspace{0.3cm}
		\begin{subfigure}[b]{0.49\textwidth}
			\centering
			\includegraphics[width=\textwidth]{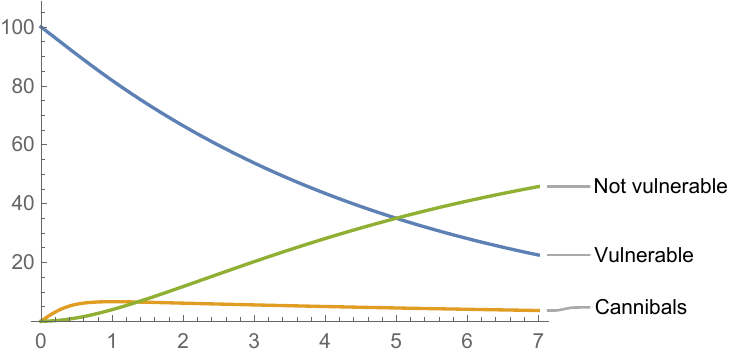}
			\caption{$\alpha_1=0.7$}
			\label{alp0.7}
		\end{subfigure}
		\caption{Variation in the time series of system (\ref{modelo}) according to the acceptance of \textit{Atractosteus tropicus} to food.}
		\label{graficas_alimento}
	\end{figure}
\end{ejem}

\begin{ejem}\label{Ejem_manejo}
	Taking $a=2$, $n=100$, $\alpha_1=0.5$, $\beta=1.69$, $\gamma=0.39$, $\mu=0.002$, $\rho_1=0.013$, and $\rho_2=0.018$ in system (\ref{modelo}) and considering two hours of management, i.e., $m=\dfrac{1}{12}$, we have a survival rate of approximately 5\%, as is shown in Figure \ref{m1_12}. If we increase it to 8 hours of management with $m=\dfrac{1}{3}$, we can see that the survival rate rises to 20\%, as shown in Figure \ref{m1_3}. By considering 16 hours of management with $m=\dfrac{2}{3}$, a survival rate of {40\%} is achieved, as shown in Figure \ref{m2_3}. In order to achieve a survival rate of {45\%}, at least 18 hours of management are required, i.e., $m=\dfrac{3}{4}$,  see Figure \ref{m3_4}. From the numerical simulations, it can be concluded that management has a significant impact on reducing cannibalism but comes at a high cost in terms of the hours of work required to invest, see Figure \ref{graficas_manejo}.
	\begin{figure}[h]
		\centering
		\begin{subfigure}[b]{0.49\textwidth}
			\centering
			\includegraphics[width=\textwidth]{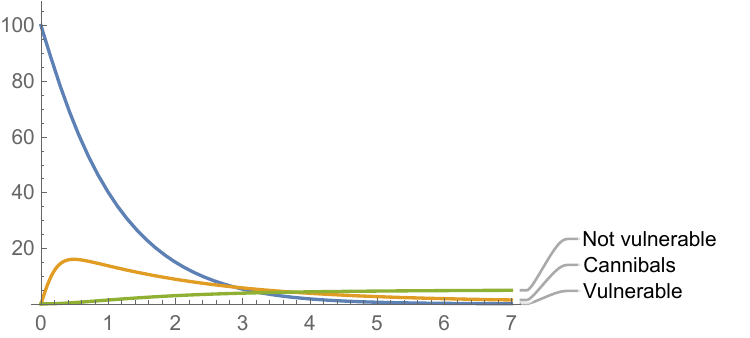}
			\caption{$m=\dfrac{1}{12}$}
			\label{m1_12}
		\end{subfigure}\vspace{0.3cm}
		\begin{subfigure}[b]{0.49\textwidth}
			\centering
			\includegraphics[width=\textwidth]{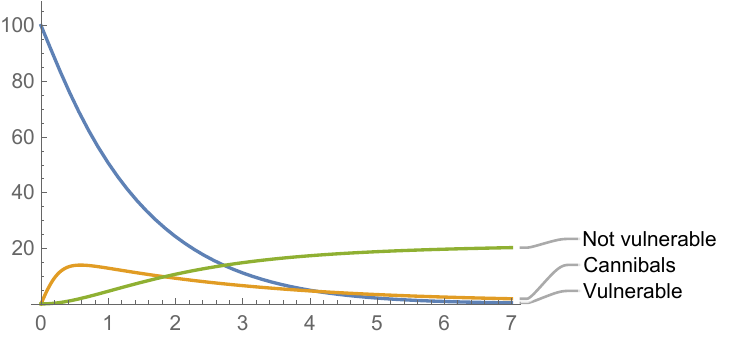}
			\caption{$m=\dfrac{1}{3}$}
			\label{m1_3}
		\end{subfigure}\vspace{0.3cm}
		\begin{subfigure}[b]{0.49\textwidth}
			\centering
			\includegraphics[width=\textwidth]{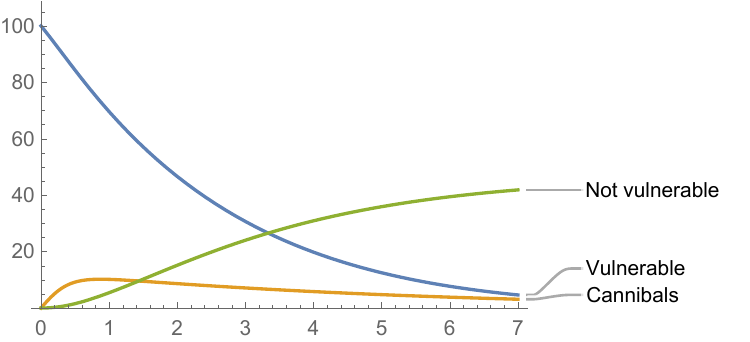}
			\caption{$m=\dfrac{2}{3}$}
			\label{m2_3}
		\end{subfigure}
		\begin{subfigure}[b]{0.49\textwidth}
			\centering
			\includegraphics[width=\textwidth]{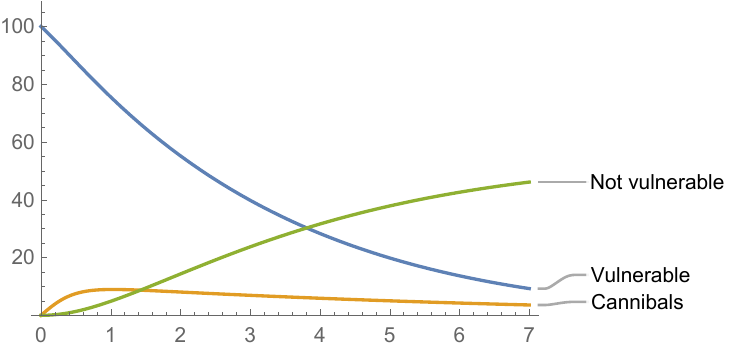}
			\caption{$m=\dfrac{3}{4}$}
			\label{m3_4}
		\end{subfigure}
		\caption{Variation in the time series of system (\ref{modelo}) according to the hours of management invested.}
		\label{graficas_manejo}
	\end{figure}
\end{ejem}
{In Figure \ref{superficies_supervivencia} we show the survival levels obtained by setting  $a=2$, $n=100$, $\alpha_1=0.5$, $m=0.5$, $\beta=1.69$, $\gamma=0.39$, $\mu=0.002$, $\rho_1=0.013$, $\rho_2=0.018$, making variations in the cannibalism rate $\beta$, in the quality of the food $\alpha_1$ and in the handling time of the cannibals $m$. In the first Figure \ref{comida} we can see how the survival percentage increases as the quality of the food improves and the rate of larvae that become cannibals decreases and in the second Figure \ref{manejo} we can see how the survival percentage increases as more time is invested in the surveillance and isolation of the cannibals, as well as the decrease in the rate of larvae that become cannibals. In both graphs 90\% survival is achieved when you have a low cannibalism rate and a high quality of food or a high handling time.\\
	In order to measure the impact of food on survival, an experiment was carried out with two populations of larvae, some that were fed and others that were not provided with food. In Figure \ref{ajuste_datos} it can be seen that survival is 100\% in the first 8 days, then cannibals begin to appear and the population begins to decrease. The difference between them is that in the one that is provided with food, survival stabilizes around 40\% on day 14, while in the other population it continues to decrease. In Fig. \ref{ajuste} we show how the solution of differential equation (\ref{modelo}) fits the experimental data. It is important to mention that the solutions of the differential equation model the drop in survival starting on day 9 when cannibalism appears, for the first days the solution is constant.
}
\begin{figure}[h]
	\centering
	\begin{subfigure}[b]{0.49\textwidth}
		\centering
		\includegraphics[width=\textwidth]{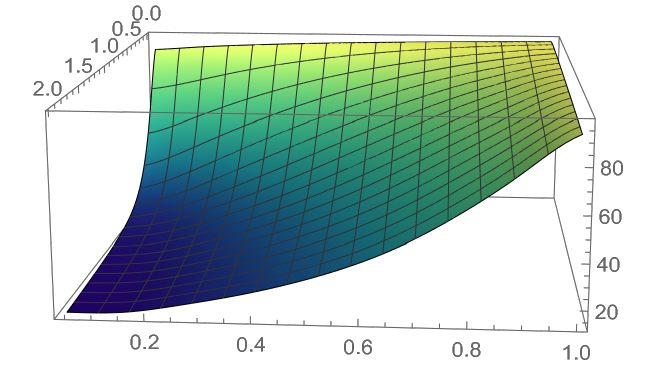}
		\put(-195,107) {{\footnotesize $\mathbf{\beta}$}}
		\put(-115,-5) {{\footnotesize $\alpha_1$}}
		\caption{{Survival fixing $m=0.5$ and  varying $\beta$ and $\alpha_1$.}}
		\label{comida}
	\end{subfigure}
	\begin{subfigure}[b]{0.49\textwidth}
		\centering
		\includegraphics[width=\textwidth]{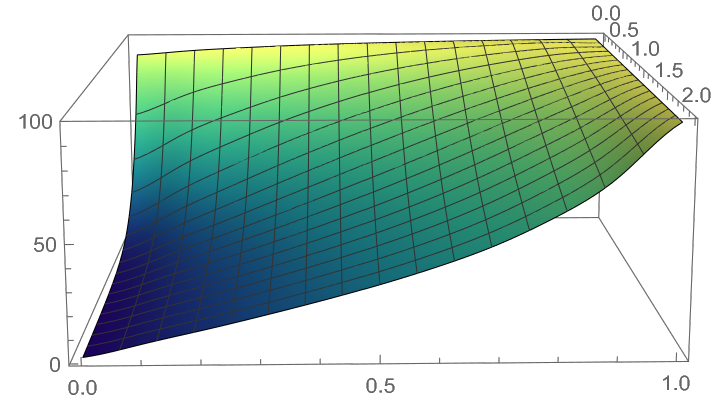}
		\put(-11,103) {{\footnotesize $\mathbf{\beta}$}}
		\put(-105,-5) {{\footnotesize $m$}}
		\caption{{Survival fixing $\alpha_1=0.5$ and  varying $\beta$ and $m$.}}
		\label{manejo}
	\end{subfigure}
	\caption{{Survival of the population setting $a=2$, $n=100$, $\beta=1.69$, $\gamma=0.39$, $\mu=0.002$, $\rho_1=0.013$ and $\rho_2=0.018$, making variations in the cannibalism rate $\beta$, food quality $\alpha_1$ and management of the cannibals $m$.}}
	\label{superficies_supervivencia}
\end{figure}

\begin{figure}[h]
	\centering
	\begin{subfigure}[b]{0.49\textwidth}
		\centering
		\includegraphics[width=\textwidth]{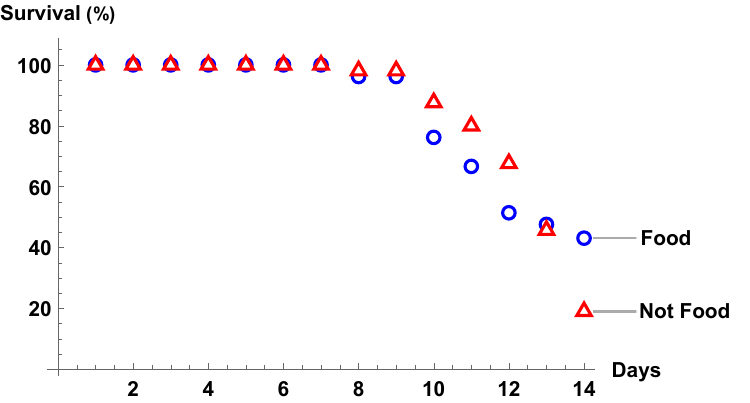}
		\caption{{Experimental data.}}
		\label{datos_reales}
	\end{subfigure}
	\begin{subfigure}[b]{0.49\textwidth}
		\centering
		\includegraphics[width=\textwidth]{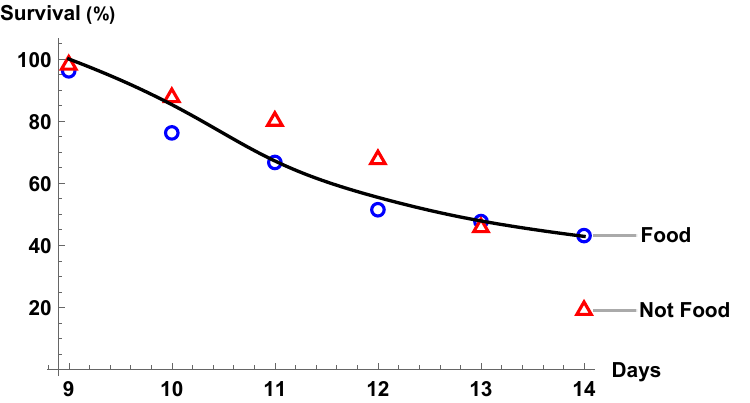}
		\caption{{Solution of model (\ref{modelo}) and experimental data.}}
		\label{ajuste}
	\end{subfigure}
	\caption{{Comparison of the experimental data with the solution of model (\ref{modelo}), fixing  $a=2$, $n=100$, $\alpha_1=0.5$, $m=0.5$, $\beta=1.69$, $\gamma=0.39$, $\mu=0.002$, $\rho_1=0.013$, $\rho_2=0.018$ and taking $(100,0,0)$ as the initial condition.}}
	\label{ajuste_datos}
\end{figure}

\section{Conclusions}
The experimental results show that in the larvae of the \textit{Atractosteus tropicus} population, the effect of cannibalism appears after the eighth day and is present until day 14, when the larval population reaches a size that inhibits the effect of cannibalism, stabilizing the survival of the species. In particular, it has been observed that in the absence of food, cannibalism is high, and survival drops to 20\%, while when food is provided to the larvae, survival has been increased to 40\%,  see Fig \ref{ajuste_datos}. The model analyzed in this study demonstrates that an improvement in the diet contributes to reducing the effect of cannibalism and increasing the survival rate. {It was also shown that an increase in the handling time of cannibals improves survival, see Fig \ref{superficies_supervivencia}. In the first eight days the solution is constant, then  the model time begins to run from the moment the presence of cannibalism exists (day 9), for this reason, Figure \ref{ajuste} only shows the solution from days 9 to 14. where day nine corresponds to time zero of the solution with initial condition (100,0,0). In this period of time, the model solutions adjust the experimental data and are also the most important in the study of the \textit{Atractosteus tropicus}  survival, since at the end of this period the  survival is  stabilized.} Therefore, even though the solutions tend asymptotically to zero, the model allows analyzing the survival of the \textit{Atractosteus tropicus} in the first fourteen days of their larval state and can be adapted to those fish species that have cannibalism in their first stages of growth. We are working in a control model that allows the evaluation and optimization of the cost of managing the cannibals and the quality of the food.

\section*{Acknowledgment}
The first author thanks CONAHCyT for the postdoctoral fellowship and Laboratorio de Fisiología en Recursos Acuaticos (LAFIRA).\\ We appreciate the comments of the reviewers, which helped us improve this work.

\section*{Appendix}
\section{Proof of Proposition \ref{CPI}}\label{A1}
\begin{proof}
	In order to prove the positive invariance of $T$, we will verify the vector field $\mathbf{V}$ given by (\ref{modelo}) in each one of the planes of the boundary of $T$.
	
	For the plane  $x=0$, we have $\mathbf{V}\cdot (1,0,0)=0$. Thus, the plane $x=0$ is invariant.\\ On the other hand, on the planes $y=0$ and $z=0$, we have
	$$\mathbf{V}\cdot (0,1,0)=(m-1)(\alpha_1-1) \beta x>0\quad\mbox{and}\quad \mathbf{V}\cdot (0,0,1)=\alpha_1 m(n-x)  (x \rho_1 + y \rho_2)>0,$$
	respectively.\\
	
	For the plane $x=n$, we have
	$$\mathbf{V}\cdot (1,0,0)=n\left(-\beta(m-1)(\alpha_1-1)+\dfrac{\beta y(\alpha_1-1)}{1+n+ay}-\mu-\rho_1\right)<0,$$
	when $m, \alpha_1 \in (0,1).$\\
	Now, on the plane $y=n$  we obtain that
	$$\mathbf{V}\cdot (0,1,0)=(m-1)(\alpha_1-1) \beta x - n (n \alpha_1 \gamma + \mu + \rho_2)<0,$$
	if $\beta<\dfrac{n\gamma \alpha_1+\mu+\rho_2}{(m-1)(\alpha_1 -1)}$.\\
	Finally, on the plane $z=n$ we have
	$$\mathbf{V}\cdot (0,0,1)=-n\mu-m\alpha_1 x(\rho_1 x+\rho_2 y)<0.$$
	
	Hence, the vector field $\mathbf{V}$ is entering each of the planes on the boundary of $T$. Therefore, the set $T$ is positively invariant for system (\ref{modelo}).
\end{proof}

\section{Proof of Theorem \ref{GS}}\label{A2}
\begin{proof}
	Let $(x_0, y_0, z_0) \in \mathbb{R}^3_+$ be an initial condition. From the first equation of the system (\ref{modelo}), we have
	$$\dot{x}\leq -C_0 x,$$
	where $C_0=(1-\alpha_1)(1-m)\beta+\rho_1+\mu>0$. Thus, we have
	\begin{equation*}
		\begin{aligned}
			\dot{x}+C_0x&\leq 0,\\
			\dfrac{d}{dt}\left(x\displaystyle e^{C_0t}\right)&\leq 0,\\
			x(t)\displaystyle e^{C_0t}&\leq x_0,
		\end{aligned}
	\end{equation*}
	Hence, we obtain that $x(t)\leq x_0e^{-C_0t}$. Therefore, $x(t)$ tends to zero, when  $t$ tends to infinity.\\
	From the second equation of (\ref{modelo}), since $x(t)$ goes to zero, there exists $t_0$ such that  $\dot{y}(t)<0$. 
	Then  $y(t)$ is decreasing for $t>t_0$. As $y(t)$ is  lower bounded, there exists $y^*$  such that  $\displaystyle\lim_{t\rightarrow\infty}y(t)=y^*$. Similarly, by the third equation, we conclude that there exists $z^*$  such that $\displaystyle\lim_{t\rightarrow\infty}z(t)=z^*$.\\
	In order to show that $y^*=0$, we assume that $y^* > 0$ and applying the theorem of existence and uniqueness to  system 
	(\ref{modelo}) with initial condition $(0, y^*, z^*)$,  we can extend the solution. Then, there exists $y_1$ in the solution of system (\ref{modelo}) passing by $(x_0,y_0,z_0)$, such that $y^* > y_1$, which is a contradiction. Therefore, $y(t)$ tends to zero as $t$ tends to infinity. 
	
	Since $x(t)$ and $y(t)$ tend to zero as $t$ approaches infinity, there exists $t_1>t_0 > 0$ such that in the third equation of (\ref{modelo})
	$$\alpha_1 m (n-x(t_1)-z(t_1))(\rho_1 x(t_1)+\rho_2 y(t_1))<\mu z(t_1),\quad \forall\, t>t_1,$$
	
	In the same way as for the function $y(t)$ we obtain that  $z(t)$ tends to zero as $t$ approaches infinity. 
	Therefore, $P_0$ is globally asymptotically stable.
\end{proof}


\section*{References}

\begin{thebibliography}{0}
\bibitem[1]{Al} Aladeen Al Basheer, Rana D. Parshad, Emmanuel Quansah, Shengbin Yu, and Ranjit Kumar Upadhyay. Exploring the dynamics of a holling-tanner model with cannibalism in both predator and prey population. \textit{International Journal of Biomathematics}, 11(01):1850010, 2018.

\bibitem[2]{Ga} V\'ictor Julio Atencio Garc\'ia and Evoy Zaniboni Filho. El canibalismo en la larvicultura de peces. \textit{Revista MVZ C\'ordoba}, 11:9–19, 2006.

\bibitem[3]{BN} E. Baras, M. Ndao, MYJ. Maxi, Denys Jeandrain, Jean-Pierre Thom\'e, P Vanadewalle, and Charles M\'elard. Sibling cannibalism in dorada under experimental conditions. i. ontogeny, dynamics, bioenergetics of cannibalism and prey size selectivity. \textit{Journal of Fish Biology}, 57(4):1001-1020, 2000.

\bibitem[4]{Ba} Etienne Baras and Malcolm Jobling. Dynamics of intracohort cannibalism in cultured fish. \textit{Aquaculture research}, 33(7):461–479, 2002.

\bibitem[5]{Bu} George C. Burg, Jaimie Johnson, Savannah Spataro, Amelia O’Keefe, Natasha Urbina, Georgina Puentedura, Matt Heuton, Sean Harris, Stanley D. Hillyard, and Frank van Breukelen. Care and propagation of captive pupfish from the genus cyprinodon: insight into conservation. \textit{Environmental Biology of Fishes}, 102:1015–1024, 2019.

\bibitem[6]{CD} Fengde Chen, Hang Deng, Zhenliang Zhu, and Zhong Li. Note on the persistence and stability property of a stage-structured prey-predator model with cannibalism and constant attacking rate. \textit{Advances in Difference Equations}, 2020(1):1-20, 2020.

\bibitem[7]{Cl} David Claessen and Andre M. de Roos. Bistability in a size-structured population model of cannibalistic fish-a continuation study. \textit{Theoretical Population Biology}, 64(1):49-65, 2003.

\bibitem[8]{Di} Odo Diekmann, RM. Nisbet, WSC. Gurney, and F. Van den Bosch. Simple mathematical models for cannibalism: a critique and a new approach. \textit{Mathematical Biosciences}, 78(1):21-46, 1986.

\bibitem[9]{Do} Shuozeng Dou, Tadahisa Seikai, and Katsumi Tsukamoto. Cannibalism in japanese flounder juveniles, paralichthys olivaceus, reared under controlled conditions. \textit{Aquaculture}, 182(1-2):149-159, 2000.

\bibitem[10]{Fo} Laurel R. Fox. Cannibalism in natural populations. \textit{Annual review of ecology and systematics}, 6(1):87–106, 1975.

\bibitem[11]{F} C.A. Fr\'ias-Quintana, C.A. \'Alvarez-Gonz\'alez, and G. M\'arquez-Couturier. Dise\~no de microdietas para el larvicultivo de pejelagarto \textit{atractosteus tropicus}, gill 1863. \textit{Universidad y ciencia}, 26(3):265-282, 2010. 

\bibitem[12]{Gg} Isabelle Gagn\'e, Daniel Coderre, and Yves Mauffette. Egg cannibalism by coleomegilla maculata lengi neonates: preference even in the presence of essential prey. \textit{Ecological Entomology}, 27(3):28-291, 2002.

\bibitem[13]{Go} Adnan Gora, K. Ambasankar, K.P. Sandeep, Saima Rehman, Deepak Agarwal, Irshad Ahamad, and Kasilingam Ramachandran. Effect of dietary supplementation of crude microalgal extracts on growth performance, survival and disease resistance of lates calcarifer (bloch, 1790) larvae. \textit{Indian Journal of Fisheries}, 66(1):64-72, 2019.

\bibitem[14]{Ha} Ron Hans, Ryan Schloesser, Nathan Brennan, Flavio Ribeiro, and Kevan L. Main. Effects of stocking density on cannibalism in juvenile common snook centropomus undecimalis. \textit{Aquaculture research}, 51(2):844–847, 2020.

\bibitem[15]{He} T. Hecht and S. Appelbaum. Observations on intraspecific aggression and coeval sibling cannibalism by larval and juvenile claias gariepinus (clariidae: Pisces) under controlled conditions. \textit{Journal of zoology}, 214(1):21-44, 1988.

\bibitem[16]{HP} Thomas Hecht and Anthony G. Pienaar. A review of cannibalism and its implications in fish larviculture. \textit{Journal of the World Aquaculture Society}, 24(2):246-261, 1993.

\bibitem[17]{KT} C. Kaewmanee and I.M. Tang. Cannibalism in an age-structured predator-preysystem. \textit{Ecological Modelling},  167(3):213–220, 2003.

\bibitem[18]{Ko} C. Kohlmeier and W. Ebenh{\"o}h. The stabilizing role of cannibalism in a predatorprey system. \textit{Bulletin of Mathematical Biology}, 57:401-411, 1995.

\bibitem[19]{La} R. Lavanya, S. Vinoth, K. Sathiyanathan, Zeric Njitacke Tabekoueng, P. Hammachukiattikul, and R. Vadivel. Dynamical behavior of a delayed holling type-ii predator-prey model with predator cannibalism. \textit{Journal of Mathematics}, 2022.

\bibitem[20]{Li} Qifa Lin, Chulei Liu, Xiangdong Xie, and Yalong Xue. Global attractivity of Leslie–Gower predator-prey model incorporating prey cannibalism. \textit{Advances in Difference Equations}, 2020(1):1-15, 2020.

\bibitem[21]{Lo} N.L. Loadman, G.E.E. Moodie, and J.A. Mathias. Significance of cannibalism in larval walleye (stizostedion vitreum). \textit{Canadian Journal of Fisheries and Aquatic Sciences}, 43(3):613-618, 1986.

\bibitem[22]{Ma} Kjartan G. Magn\'usson. Destabilizing effect of cannibalism on a structured predator–prey system. \textit{Mathematical biosciences}, 155(1):61-75, 1999.

\bibitem[23]{Mc} G. M\'arquez-Couturier. \textit{Acuicultura tropical sustentable: una estrategia para la producci\'on y conservaci\'on del pejelagarto (\textit{Atractosteus tropicus}) en Tabasco, M\'exico}. UJAT, 2015.

\bibitem[24]{Mc1} G. M\'arquez-Couturier. Estado de arte de la biolog\'ia y cultivo de pejelagarto (\textit{atractosteus tropicus}). \textit{Agro Productividad}, 8(3), 2015.

\bibitem[25]{N} J.S. Nelson. \textit{Fishes of the world 4th edition} John Wiley \& sons. Nueva York, 2006.

\bibitem[26]{P} D.J. Palma-Cancino, R. Mart\'inez-Garc\'ia, C.A. \'Alvarez-Gonz\'alez, S. Camarillo-Coop, and E.S. Pe{\~n}a-Mar\'in. Esquemas de alimentaci\'on para larvicultura de pejelagarto (\textit{atractosteus tropicus gill}): crecimiento, supervivencia y canibalismo. \textit{Ecosistemas y recursos agropecuarios}, 6(17):273-281, 2019.

\bibitem[27]{Pe} Larissa Strictar Pereira, Angelo Antonio Agostinho, and Kirk O. Winemiller. Revisiting cannibalism in fishes. \textit{Reviews in fish biology and fisheries}, 27:499-513, 2017.

\bibitem[28]{Pf} David W. Pfennig. Kinship and cannibalism. \textit{Bioscience}, 47(10):667-675, 1997. 

\bibitem[29]{Pi} Anthony Graham Pienaar. A study of coeval sibling cannibalism in larval and juvenile fishes and its control under culture conditions. \textit{Rhodes University; Faculty of Science, Zoology and Entomology}, 1990.

\bibitem[30]{Po} Gary A. Polis. The evolution and dynamics of intraspecific predation. \textit{Annual Review of Ecology and Systematics}, 12(1):225-251, 1981.

\bibitem[31]{Qi} Jianguang Qin and Arlo W. Fast. Size and feed dependent cannibalism with juvenile snakehead channa striatus. \textit{Aquaculture}, 144(4):313-320, 1996.

\bibitem[32]{Re} R.E. Reis, J.S. Albert, F. Di Dario, M.M. Mincarone, P. Petry, and L.A. Rocha. Fish biodiversity and conservation in South America. \textit{Journal of fish biology}, 89(1):12-47, 2016.

\bibitem[33]{Ro} Sievert Rohwer. Parent cannibalism of offspring and egg raiding as a courtship strategy. \textit{The American Naturalist}, 112(984):429-440, 1978.

\bibitem[34]{Se} C.A. Sep\'ulveda-Quiroz, C.S. \'Alvarez-Villagomez, O. Mendoza-Porras, E.S. Pe{\~n}a-Mar\'in, C.I. Maytorena-Verdugo, G.M. P\'erez-Jim\'enez, R. Jesus-Contreras, C.A. \'Alvarez-Gonz\'alez, and R. Mart\'inez-Garc\'ia. Attack behavior leading cannibalism in tropical gar (\textit{atractosteus tropicus}) larvae under different tank colors and shelter type. \textit{Aquaculture}, 563:738991, 2023.

\bibitem[35]{Sm} Carl Smith and Peter Reay. Cannibalism in teleost fish. \textit{Reviews in Fish Biology and Fisheries}, 1:41-64, 1991.

\bibitem[36]{Sw} J. Sweetman. Managerial procedures and cost effectiveness in larval rearing. \textit{Spec Pub Eur Aquacult Soc}, 15:385, 1991.

\bibitem[37]{Ta} E. M. Takyi, K. Cooper, A. Dreher, and C. McCrorey. The (de) stabilizing effect of juvenile prey cannibalism in a stage-structured model. \textit{Mathematical Biosciences and Engineering}, 30:3355-3378, 2023.

\bibitem[38]{Va} P. Van Damme, S. Appelbaum, and T. Hecht. Sibling cannibalism in koi carp, cypvinus carpio l., larvae and juveniles reared under controlled conditions. \textit{Journal of Fish Biology}, 34:855-863, 1989.

\bibitem[39]{Za} F. Zhang, Y. Chen, and J. Li. Dynamical analysis of a stage-structured predatorprey model with cannibalism. \textit{Mathematical Biosciences}, 34:855-863, 33-41.

\end{thebibliography}
\end{document}